\documentclass[preprint,12pt]{article}
\usepackage{float}
\usepackage{amssymb}
\usepackage{mathrsfs}
\usepackage{enumerate}
\usepackage{subfigure}
\usepackage{graphicx}
\usepackage{colortbl,dcolumn}
\usepackage{tabularx}
\usepackage{amsmath}
\usepackage{psfrag}
\usepackage{booktabs}
\usepackage{cite}
\usepackage{amsthm}

\numberwithin{equation}{section}
\usepackage[top=1.2in, bottom=1.2in, left=1in, right=1in]{geometry}

\newcommand{\R}{\mathbb{R}}

\newtheorem{tm}{Theorem}[section]
\newtheorem{df}{Definition}

\newtheorem{rk}{Remark}

\begin{document}
\title{Structure-Preserving Implicit Runge-Kutta Methods for Stochastic Poisson Systems with Multiple Noises}
       \author{
        Liying Zhang\footnotemark[1], Fenglin Xue\footnotemark[2], Lijin Wang\footnotemark[3] \\
        {\small \footnotemark[1]~\footnotemark[2] School of Mathematical Science, China University of Mining and Technology, Beijing 100083, China} \\
        {\small \footnotemark[3] School of Mathematical Sciences, University of Chinese Academy of Sciences, Beijing 100049, China} 
    }
       \maketitle
        \footnotetext{\footnotemark[2]Corresponding author: xuefenglinlin@163.com}

       \begin{abstract}
          {\rm\small In this paper, we propose the diagonal implicit Runge-Kutta methods and transformed Runge-Kutta methods for stochastic Poisson systems with multiple noises. We prove that the first methods can preserve the Poisson structure, Casimir functions, and quadratic Hamiltonian functions in the case of constant structure matrix. Darboux-Lie theorem combined with coordinate transformation is used to construct the transformed Runge-Kutta methods for the case of non-constant structure matrix that preserve both the Poisson structure and the Casimir functions. Finally, through numerical experiments on stochastic rigid body systems and linear stochastic Poisson systems, the structure-preserving properties of the proposed two kinds of numerical methods are effectively verified.}\\

\textbf{Key words: }{\rm\small}stochastic Poisson systems, diagonal implicit Runge-Kutta methods, transformed Runge-Kutta methods, Poisson structure, Casimir functions
\end{abstract}

\section{Introduction}
\label{int}

The classical Hamiltonian systems are following form
\begin{equation}\label{1.1}
	dy = J^{-1}\nabla H(y(t))dt, \quad y \in \mathbb{R}^{2n},
\end{equation}
where $J = \begin{pmatrix}
	0 & I_n\\
	-I_n & 0
\end{pmatrix}$ is a symplectic matrix, and $H(y(t)) \in C^\infty(\mathbb{R}^{2n})$ is the Hamiltonian function. Hamiltonian systems are a class of very important mechanical systems defined on manifolds of even dimensions (see \cite{Haireretal2006} et al.). Sophus Lie \cite{Haireretal2006} proposed the Poisson system in 1888, which is an extension of the Hamiltonian system \eqref{1.1}, with the following form:
\begin{equation}\label{1.2}
	dy = B(y(t))\nabla H(y(t))dt.
\end{equation}
The commonality between these two systems lies in the fact that properties of the structure matrices $J$ and $B(y)$ satisfy the skew-symmetry condition \eqref{2.2} and the Jacobi identity \eqref{2.3}. The difference is that in system  \eqref{1.2} the dimension of $y$ is not restricted, but in \eqref{1.1} $y$ is required to be even dimensional.
Therefore, this allows the Poisson system to describe a broad range of dynamical systems because it widens the phase space dimensionality constraints. The Poisson systems have broad applicability and show significant value in various fields, including the astronomy, the biomechanical research, the fluid mechanics analysis and so on (see \cite{Fengetal2010,Haireretal2006,Zhuetal1994} et al.).

Influenced by stochastic factors, stochastic Hamiltonian systems and stochastic Poisson systems are introduced and investigated. Currently, structure-preserving methods have emerged as a focal point of research within the realm of stochastic Hamiltonian systems, leading to significant advancements  (see \cite{Hongetal2015,Hongetal2023,Milsteinetal2002a,Milsteinetal2002b,Misawaetal2000,Wangetal2017} et al.). As a generalization of the stochastic Hamiltonian systems, stochastic Poisson systems also exhibit structure-preserving properties, including Poisson structure, Casimir functions, and energy. If a numerical method almost surely satisfies the Poisson structure and Casimir functions, it is called a stochastic Poisson integrator for the stochastic Poisson system (see\cite{Hongetal2021}). In recent years, there have been studies on constructing stochastic Poisson integrators for stochastic Poisson systems driven by multiple noises. For instance, \cite{Wangetal2021} based on the Pad\'e approximation constructed a class of stochastic Poisson integrators for linear stochastic Poisson systems with constant structure matrix. \cite{Wangetal2022a} employed the midpoint method to construct stochastic Poisson integrators for the same stochastic Poisson systems. Additionally, \cite{Hongetal2021} proposed stochastic Poisson integrators for the original stochastic Poisson systems with non-constant structure matrix based on the Darboux-Lie theorem and $\alpha$-generating functions. On the basis of \cite{Hongetal2021}, 
the authors introduced the transformed midpoint method to construct stochastic Poisson integrators for the original stochastic Lotka-Volterra systems (see \cite{Liuetal2024}).

Besides preserving stochastic Poisson integrators, some scholars have focused on numerical methods that preserve both the Casimir functions and the energy. For example, \cite{Cohenetal2014} constructed a numerical method exactly preserving both the energy and quadratic Casimir functions for stochastic Poisson systems driven by single noise with non-constant structure matrix. Following this work, \cite{Wangetal2022a} developed the energy-Casimir-preserving scheme that preserves the energy and the general Casimir functions. In addition, there are some research works specifically only dedicated to energy-preserving numerical methods. \cite{Lietal2019} proposed a family of explicit parameterized stochastic Runge-Kutta methods for stochastic Poisson systems driven by single noise with non-constant structure matrix and proved that these methods can be energy-preserving with appropriate parameters. Moreover, \cite{Wangetal2022b} constructed the Hamiltonians-preserving schemes designed to preserve all Hamiltonian functions simultaneously for stochastic Poisson systems driven by multiple noises with non-constant structure matrix, drawing on the averaged vector field discrete gradient and orthogonal projection methods.

Up to now, there has been no reference on implicit Runge-Kutta methods for stochastic Poisson systems driven by multiple noises that can simultaneously preserve the Poisson structure, the Casimir functions, and the quadratic Hamiltonian functions. Inspired by pioneering works, our work constructs two classes of structure-preserving numerical methods. The first class of methods is diagonal implicit Runge-Kutta methods for stochastic Poisson systems with a constant structure matrix. The methods are obtained by the generalized midpoint method, and then combine the stochastic characteristics of the diffusion coefficients and employ variable coefficient techniques. We find that the coefficients of the diagonal implicit Runge-Kutta methods are the same as the conditions of the symplectic Runge-Kutta methods. The obtained coefficient conditions can overcome the difficulties arising from the structure matrix. We prove that the diagonal implicit Runge-Kutta methods preserve the Poisson structure, Casimir functions, and quadratic Hamiltonian functions of the stochastic Poisson systems. The second class of methods is transformed Runge-Kutta methods, which are constructed by utilizing the Darboux-Lie theorem for coordinate transformation. The key point is that we convert stochastic Poisson systems with a non-constant structure matrix into generalized stochastic Hamiltonian systems. Then we apply implicit Runge-Kutta methods for the generalized stochastic Hamiltonian systems. Subsequently, the transformed Runge-Kutta methods are acquired through transforming back to the original stochastic Poisson systems. It is show that the methods preserve the Poisson structure and the Casimir functions. As an application, these methods are applied for numerically solving three-dimensional stochastic rigid body systems and three-dimensional linear stochastic Poisson systems. The numerical results indicate that the proposed methods exhibit long-term computational stability and efficiency.

This paper is organized as follows. In Section 2, we provide a brief introduction to the basic knowledge of stochastic Poisson systems. In Section 3, we describe the construction process of two classes of numerical methods. In Section 4, we prove the structure-preserving properties of these methods. Two numerical experiments to demonstrate our theoretical results are presented in Section 5. The final section provides a comprehensive summary of our work.

\section{Stochastic Poisson systems}
\label{SPS}
Stochastic Poisson systems in $\mathbb{R}^d$ with initial value $y_0$ at $t_0$ are described as follows (see \cite{Hongetal2021}):
\begin{equation}\label{2.1}
	\begin{split}
		&dy(t)= B(y(t))\left(\nabla H_0(y(t))dt + \sum_{r=1}^m\nabla H_r(y(t))\circ dW_r(t)\right),\\
		&y(t_0)= y_0, \quad t\in[t_0,T],
	\end{split}
\end{equation}
where $H_i(y)\colon \mathbb{R}^d \rightarrow \mathbb{R} \, (i=0, \cdots, m)$ are smooth Hamiltonian functions, $W(t)=(W_1, \cdots, W_m)$ is an $m$-dimensional standard Wiener process, and "$\circ$" denotes the Stratonovich product. In this context, $B(y) = (b_{ij}(y)) \in \mathbb{R}^{d \times d}$ is a skew-symmetric matrix
\begin{equation}\label{2.2}
	b_{ij}(y)=-b_{ji}(y)
\end{equation}
satisfying the Jacobi identity
\begin{equation}\label{2.3}
	\sum_{s=1}^d\left(\frac{\partial b_{ij}(y)}{\partial y_s}b_{sk}(y)+ \frac{\partial b_{jk}(y)}{\partial y_s}b_{si}(y)+\frac{\partial b_{ki}(y)}{\partial y_s}b_{sj}(y)\right)=0, \quad \text{for} \quad i, j, k=1, \ldots, d.
\end{equation}

We start with a smooth matrix-valued function $B(y)$, by which the Poisson bracket between two functions $F$ and $G$ is defined as
\begin{equation}\label{2.4}
	\left\{F, G \right\}(y):=\left(\nabla F(y) \right)^\top B(y) \nabla G(y), \quad \forall F, G \in C^\infty(\mathbb{R}^d).
\end{equation}
It can be proven that the Poisson bracket defined above satisfies the following properties (see reference \cite{Fengetal2010} for details):
\begin{enumerate}[(i)]
	\item Skew-symmetry:
	\begin{equation*}
		\left\{F, G \right\}=-\left\{G, F \right\};
	\end{equation*}
	\item Bilinearity:
	\begin{equation*}
		\begin{aligned}
			\left\{\alpha F+\beta G, H \right\}&=\alpha \left\{F, H \right\}+\beta \left\{G, F  \right\};\\
			\left\{H, \alpha F+\beta G \right\}&=\alpha \left\{H, F \right\}+\beta \left\{H, G  \right\};
		\end{aligned}
	\end{equation*}
	\item Jacobi identity:
	\begin{equation*}
		\left\{\left\{F, G \right\}, H \right\}+ \left\{\left\{G, H \right\}, F \right\}+\left\{\left\{H, F \right\}, G \right\}=0;
	\end{equation*}
	\item Leibniz rule:
	\begin{equation*}
		\left\{F\cdot G, H \right\}=F \cdot \left\{G, H \right\}+G \cdot \left\{F, H \right\}.
	\end{equation*}
\end{enumerate}

Based on Poisson bracket \eqref{2.4}, we can define the following Poisson mapping, and its equivalent form is presented through Theorem \ref{th2.1}.
\begin{df}\label{ef1}(\cite{Fengetal2010})
	A transformation $\varphi \colon \mathbb{R}^d \to \mathbb{R}^d$ is called a Poisson mapping, if it preserves the Poisson bracket, i.e., 
		\begin{equation}\label{2.5}
		\left\{F\circ \varphi, G\circ \varphi \right\}(y)=\left\{F, G \right\}\circ \varphi(y), \quad \forall F, G \in C^\infty(\mathbb{R}^d).
	\end{equation}
\end{df}

\begin{tm}\label{th2.1}(\cite{Fengetal2010})
	For a Poisson manifold with structure matrix $B(y)$, \eqref{2.5} is equivalent to
		\begin{equation*}
		\frac{\partial \varphi (y)}{\partial y}B(y){\left(\frac{\partial \varphi (y)}{\partial y}\right)}^\top=B(\varphi (y)).
	\end{equation*}
\end{tm}
Furthermore, if the structure matrix $B(y)$ of a stochastic Poisson system \eqref{2.1} is not full rank, then there exist Casimir functions in the system. Subsequently, we give the specific definition of the Casimir functions.
\begin{df}\label{ef2}(\cite{Haireretal2006})
	A function $C(y)$ is called a Casimir function of the stochastic Poisson system (\ref{2.1}), if
		\begin{equation*}
		\nabla C(y)^\top B(y)=0, \quad \forall y\in \mathbb{R}^d.
	\end{equation*}
\end{df}

It is easy to show that the Casimir function $C(y)$ is an invariant of the stochastic Poisson system \eqref{2.1}, see reference \cite{Hongetal2021} for more details.

A numerical method $\{y_n\}$ preserving the Poisson structure as well as the Casimir function for stochastic Poisson systems \eqref{2.1}, namely,
\begin{equation*}
	\frac{\partial y_{n+1}}{\partial y_n}B(y_n){\left(\frac{\partial y_{n+1}}{\partial y_n}\right)}^\top=B(y_{n+1}), \quad \forall n\in \mathbb{N},\quad a.s.,
\end{equation*}
\begin{equation*}
	C(y_{n+1})=C(y_n), \quad \forall n\in \mathbb{N},\quad a.s.,
\end{equation*}
are called stochastic Poisson integrators (see \cite{Hongetal2021}).

\section{Runge-Kutta methods for stochastic Poisson systems}
\subsection{Transformed Runge-Kutta methods}
\label{sec3.1}
	 The $s$-stage implicit Runge-Kutta methods with time step $h$ applied to the stochastic Poisson systems (\ref{2.1}) have the following form (see \cite{Burrage1999}):
	\begin{equation}\label{3.1}
		\begin{aligned}
			Y_i &= y_k+h\sum_{j=1}^sa_{ij}^0f_0(Y_j)+\sum_{r=1}^m\sum_{j=1}^sa_{ij}^rf_r(Y_j)J_k^r,\quad i=1,  \cdots, s,\\
			y_{k+1} &=y_k+h\sum_{i=1}^sb_{i}^0f_0(Y_i)+\sum_{r=1}^m\sum_{i=1}^sb_{i}^rf_r(Y_i)J_k^r,
		\end{aligned}
	\end{equation}
where $f_l(y)=B(y)\nabla H_l(y)\, (l=0, \cdots, m)$, $J_k^r=W_r(t_{k+1})-W_r(t_k) (r=1, \cdots, m)$. If $A=(a_{ij}^0)$ and $B^r=(a_{ij}^r)$ are $s \times s$  matrices of real elements, $(b^0)^\top =(b_1^0, b_2^0,\cdots, b_s^0)$ and $(b^r)^\top =(b_1^r, b_2^r,\cdots, b_s^r)$ are $s$-dimensional row vectors, then methods in \eqref{3.1} can be represented by a Butcher tableau
\begin{equation*}
	\renewcommand{\arraystretch}{1.5} 
	\begin{array}{|ccccc}
		A & B^1 & B^2 & \cdots & B^r \\
		\hline
		(b^0)^\top & (b^1)^\top & (b^2)^\top & \cdots & (b^r)^\top
	\end{array}.
\end{equation*}

The reference \cite{Wangetal2017} proved that the stochastic Runge-Kutta methods \eqref{3.1} preserve the symplectic structure almost surely, as explicitly stated in Theorem \ref{th3.1}.
\begin{tm}\label{th3.1}(\cite{Wangetal2017})
	For stochastic Hamiltonian systems (\ref{3.4}), if the coefficients of \eqref{3.1} satisfy
	\begin{equation}\label{3.2}
		\begin{aligned}
			b_i^0b_j^0- b_i^0a_{ij}^0- b_j^0a_{ji}^0&= 0,\\
			b_i^0b_j^r- b_i^0a_{ij}^r- b_j^ra_{ji}^0&=0,\\
			b_i^rb_j^\zeta- b_i^ra_{ij}^\zeta- b_j^\zeta a_{ji}^r&=0
		\end{aligned}
	\end{equation}
	for all $i, j=1, \cdots, s; r, \zeta=1, \cdots, m$, then the methods preserve the symplectic structure almost surely.
\end{tm}
Below we introduce the Darboux-Lie theorem, which is used to prove that stochastic Poisson systems \eqref{2.1} preserve the Poisson structure.
\begin{tm}[{\textbf{Darboux-Lie theorem}}]\label{th3.2}(\cite{Haireretal2006})
	Suppose that the matrix $B(y)$ defines a Poisson bracket and is of constant rank $d-l=2n$ in a neighborhood of $y_0\in\R^d$. Then there exist functions $P_1(y),\cdots,P_n(y),Q_1(y),\cdots,Q_n(y)$, and $C_1(y),\cdots,C_l(y)$ satisfying 
	\begin{equation}\label{3.3} 
		\begin{matrix}
			&\{P_i,P_j\}=0, &\{P_i,Q_j\}=-\delta_{ij}, &\{P_i,C_s\}=0, \\
			&\{Q_i,P_j\}=\delta_{ij}, &\{Q_i,Q_j\}=0, &\{Q_i,C_s\}=0,\\
			&\{C_k,P_j\}=0, &\{C_k,Q_j\}=0, &\{C_k,C_s\}=0
		\end{matrix}
	\end{equation}
	for $i=1,\ldots,n,j=1,\ldots,n,k=1,\ldots,l,s=1,\ldots,l$, on a neighborhood of $y_0$. The gradients of $P_i,\,\,Q_j,\,\,C_k$ are linearly independent, so that the $\R^d\rightarrow\R^d$ mapping $y\rightarrow(P_i(y), Q_j(y), C_k(y))$ constitutes a local change of coordinates to canonical form.
\end{tm}
The authors proposed stochastic Poisson integrators for stochastic Poisson systems (SPSs) by first converting them to generalized stochastic Hamiltonian systems (SHSs) via coordinate transformation implied by the Darboux-Lie theorem, then applying stochastic symplectic methods to the SHSs, and finally converting the symplectic methods back to the original SPSs to get stochastic Poisson integrators (see \cite{Hongetal2021}). In this paper, we apply the symplectic Runge-Kutta methods to the generalized SHSs transformed from the original SPSs. Then we obtain the transformed stochastic Runge-Kutta methods for the SPSs, which possess structure-preserving properties. The specific procedure is as follows:
\begin{itemize}
	\item  By the Darboux-Lie theorem, we find a coordinate transformation $\theta(y) \colon y\rightarrow \bar{y}=(Z(y)^\top,\mathcal{C}(y)^\top)^\top$ and $Z(y)=(P(y)^\top,Q(y)^\top)^\top$ with $P(y)=(P_1(y),\cdots,P_n(y))^\top,$ $Q(y)=(Q_1(y),\cdots,Q_n(y))^\top,$ $\mathcal{C}(y)=(C_1(y),\cdots,C_l(y))^\top$, which transforms the system \eqref{2.1} with initial value $y_0$ into the following SHSs with initial value $Z_0=(P(y_0)^\top, Q(y_0)^\top)^\top$,
	\begin{equation}\label{3.4}
		\begin{split}
		dZ&=J^{-1}\left(\nabla_Z K_0(Z, \mathcal{C})dt+\sum_{r=1}^m\nabla_Z K_r(Z, \mathcal{C})\circ dW_r(t)\right),\\
		d\mathcal{C}&=0,
		\end{split}
	\end{equation}
where $K_l(Z, \mathcal{C})=K_l(\bar{y})=H_l(y) (l=0, \cdots,m),$ $J^{-1}=\begin{pmatrix}
	0 & -I_n \\
	I_n & 0 
\end{pmatrix}$.
	\item 
Apply the symplectic Runge-Kutta methods \eqref{3.1} to \eqref{3.4}
\begin{equation*}
	\begin{aligned}
		\bar{Y}_i &= \bar{y}_k+h\sum_{j=1}^sa_{ij}^0g_0(\bar{Y}_j)+\sum_{r=1}^m\sum_{j=1}^sa_{ij}^rg_r(\bar{Y}_j)J_k^r, \\
		\bar{y}_{k+1}&=\bar{y}_k+ h\sum_{i=1}^sb_i^0g_0(\bar{Y}_i)+\sum_{r=1}^m\sum_{i=1}^sb_i^rg_r(\bar{Y}_i)J_k^r,
	\end{aligned}
\end{equation*}
where $g_l(\bar{y})
= \begin{pmatrix}
	J^{-1} & 0\\
	0 & 0
\end{pmatrix} \nabla K_l(\bar{y})=\frac{\partial \bar{y}}{\partial y}B(y)\nabla H_l(y) =\theta_yf_l(y)$. Let $g_l(\bar{y})\colon =v_l(y)$, and combined with $\bar{y}=\theta(y)$, we know that
\begin{equation*}
	\begin{aligned}
		\theta(Y_i)&=\theta(y_k)+h\sum_{j=1}^sa_{ij}^0v_0(Y_j)+\sum_{r=1}^m\sum_{j=1}^sa_{ij}^rv_r(Y_j)J_k^r, \\
		\theta(y_{k+1})&=\theta(y_k)+h\sum_{i=1}^sb_i^0v_0(Y_i)+\sum_{r=1}^m\sum_{i=1}^sb_i^rv_r(Y_i)J_k^r.
	\end{aligned}
\end{equation*}
	\item Using the inverse transformation $y_{k+1}=\theta^{-1}(\bar{y}_{k+1})$, we obtain the numerical solutions $y_{k+1}$ for the original system \eqref{2.1}, 
	\begin{equation}\label{3.5}
		\begin{aligned}
			Y_i&=\theta^{-1}\left(\theta(y_k)+h\sum_{j=1}^sa_{ij}^0v_0(Y_j)+\sum_{r=1}^m\sum_{j=1}^sa_{ij}^rv_r(Y_j)J_k^r\right), \\
			y_{k+1}&=\theta^{-1}\left(\theta(y_k)+h\sum_{i=1}^sb_i^0v_0(Y_i)+\sum_{r=1}^m\sum_{i=1}^sb_i^rv_r(Y_i)J_k^r\right).
		\end{aligned}
	\end{equation}
\end{itemize}
\begin{rk}\label{rk1}
	When $\theta_l$ is the identity mapping, the numerical methods \eqref{3.5} reduce to the stochastic Runge-Kutta methods \eqref{3.1}. 
\end{rk}
\subsection{Diagonal implicit Runge-Kutta methods}
\label{sec3.2}

By generalizing the midpoint method and considering the randomness of the diffusion coefficients, we propose the diagonal implicit Runge-Kutta methods as a composite of $s$ midpoint methods using variable coefficient techniques. The following theorem elaborates on this conclusion.

\begin{tm}\label{th3.3}
	The diagonal implicit Runge-Kutta methods can be written in the following form
\begin{equation}\label{3.6}
	\begin{array}{|ccccccccccccc}
		 \frac{b_1^{0}}{2} &   &   & &\frac{b_1^{1}}{2} &   &  &  &\cdots  &\frac{b_1^{m}}{2} &   &  &\\
		 b_1^{0} & \frac{b_2^{0}}{2} &  &  & b_1^{1} & \frac{b_2^{1}}{2} & & &\cdots  & b_1^{m} & \frac{b_2^{m}}{2} & &\\
		 \vdots & \vdots & \ddots & 	& \vdots & \vdots & \ddots & &\cdots 	& \vdots & \vdots & \ddots &\\
		 b_1^{0} & b_2^{0} & \cdots & \frac{b_s^{0}}{2} & b_1^{1} & b_2^{1} & \cdots & \frac{b_s^{1}}{2} &\cdots & b_1^{m} & b_2^{m} & \cdots & \frac{b_s^{m}}{2}\\
		\hline
		 b_1^{0} & b_2^{0} & \cdots & b_s^{0} & b_1^{1} & b_2^{1} & \cdots & b_s^{1}&\cdots  & b_1^{m} & b_2^{m} & \cdots & b_s^{m} \\
	\end{array}
\end{equation}
\end{tm}
where $b_s^l=1-b_1^l- \cdots-b_{s-1}^l$ for $l=0, \cdots ,m.$
\begin{proof}
Assuming that the midpoint method is applied over $s$ steps with time steps $b_1^{0}h, b_2^{0}h, \cdots,b_s^{0}h$ respectively, and the diffusion coefficients are taken as $b_i^{r}f_r, \, r=1, \cdots, m, \, i=1, \cdots, s$.
	\begin{equation*}
		y_k\stackrel{b_1^{0}h}{\longrightarrow}y_{k+\frac{1}{s}}\stackrel{b_2^{0}h}{\longrightarrow}y_{k+\frac{2}{s}}\cdots \stackrel{b_s^{0}h}{\longrightarrow}y_{k+1}
	\end{equation*}
then there are
	\begin{equation}\label{3.7}
		\begin{aligned}
			y_{k+\frac{1}{s}} &= y_k + b_1^0 h f_0\left(\frac{y_{k+\frac{1}{s}} + y_k}{2}\right) + \sum_{r=1}^m b_1^r f_r \left(\frac{y_{k+\frac{1}{s}} + y_k}{2}\right) J_k^r
			,\\
			y_{k+\frac{2}{s}} &= y_{k+\frac{1}{s}} + b_2^0 h f_0\left(\frac{y_{k+\frac{2}{s}}+y_{k+\frac{1}{s}}}{2}\right) + \sum_{r=1}^m b_2^r f_r \left(\frac{y_{k+\frac{2}{s}}+y_{k+\frac{1}{s}}}{2}\right) J_k^r
			,\\
			&\cdots\\
			y_{k+1} &= y_{k+\frac{s-1}{s}} + b_s^0 h f_0\left(\frac{y_{k+1}+y_{k+\frac{s-1}{s}}}{2}\right) + \sum_{r=1}^m b_s^r f_r \left(\frac{y_{k+1}+y_{k+\frac{s-1}{s}}}{2}\right) J_k^r.
		\end{aligned}
	\end{equation}
Transform the formula \eqref{3.7} into the following form
	\begin{equation}\label{3.8}
		\begin{aligned}
			\frac{y_{k+\frac{1}{s}} + y_k}{2} &= y_k + \frac{b_1^0}{2} h f_0\left(\frac{y_{k+\frac{1}{s}} + y_k}{2}\right) + \sum_{r=1}^m \frac{b_1^r}{2} f_r \left(\frac{y_{k+\frac{1}{s}} + y_k}{2}\right) J_k^r
			,\\
			\frac{y_{k+\frac{2}{s}}+y_{k+\frac{1}{s}}}{2} &= y_{k+\frac{1}{s}} + \frac{b_2^0}{2} h f_0\left(\frac{y_{k+\frac{2}{s}}+y_{k+\frac{1}{s}}}{2}\right) + \sum_{r=1}^m \frac{b_2^r}{2} f_r \left(\frac {y_{k+\frac{2}{s}}+y_{k+\frac{1}{s}}}{2}\right) J_k^r
			,\\
			&\cdots\\
			\frac{y_{k+1}+y_{k+\frac{s-1}{s}}}{2} &= y_{k+\frac{s-1}{s}} + \frac{b_s^0}{2} h f_0\left(\frac{y_{k+1}+y_{k+\frac{s-1}{s}}}{2}\right) + \sum_{r=1}^m \frac{b_s^r}{2} f_r \left(\frac{y_{k+1}+y_{k+\frac{s-1}{s}}}{2}\right) J_k^r.
		\end{aligned}
	\end{equation}
	Let
	\begin{equation}\label{3.9}
		\begin{aligned}
			Y_1=&\frac{y_{k+\frac{1}{s}} + y_k}{2} ,\\
			Y_2=&\frac{y_{k+\frac{2}{s}}+y_{k+\frac{1}{s}}}{2},\\
			&\cdots\\
			Y_s=&\frac{y_{k+1}+y_{k+\frac{s-1}{s}}}{2} .
		\end{aligned}
	\end{equation}
Combining \eqref{3.7}, \eqref{3.8} and \eqref{3.9}, we have 
	\begin{equation}\label{3.10}
		\begin{aligned}
			Y_1 =& y_k + \frac{b_1^0}{2} h f_0\left(Y_1 \right) + \sum_{r=1}^m \frac{b_1^r}{2} f_r \left(Y_1 \right) J_k^r
			,\\
			Y_2 =& y_k + b_1^0hf_0\left(Y_1 \right)+\frac{b_2^0}{2} h f_0\left(Y_2 \right) +\sum_{r=1}^m b_1^r f_r \left(Y_1 \right) J_k^r+ \sum_{r=1}^m \frac{b_2^r}{2} f_r \left(Y_2 \right) J_k^r
			,\\
			&\cdots\\
			Y_s =& y_k +  b_1^0hf_0\left(Y_1 \right)+ b_2^0hf_0\left(Y_2 \right)+\cdots +  \frac{b_s^0}{2} h f_0\left(Y_s \right)\\
			&+\sum_{r=1}^m b_1^r f_r \left(Y_1 \right) J_k^r+\sum_{r=1}^m b_2^r f_r \left(Y_2 \right) J_k^r+\cdots + \sum_{r=1}^m \frac{b_s^r}{2} f_r \left(Y_s \right) J_k^r,
		\end{aligned}
	\end{equation}
\begin{equation*}
	\begin{aligned}
		y_{k+1}=&y_k+b_1^0hf_0\left(Y_1 \right)+ b_2^0hf_0\left(Y_2 \right)+\cdots + b_s^0 h f_0\left(Y_s \right)\\
		&+\sum_{r=1}^m b_1^r f_r \left(Y_1 \right) J_k^r+\sum_{r=1}^m b_2^r f_r \left(Y_2 \right) J_k^r+\cdots + \sum_{r=1}^m b_s^r f_r \left(Y_s \right) J_k^r.
	\end{aligned}
\end{equation*}
\end{proof}
\begin{rk}\label{rk2}
The diagonal implicit Runge-Kutta methods \eqref{3.6} preserve the symplectic structure because the coefficients $a_{ij}^l$ are given by $a_{ij}^l = \frac{1}{2}b_j^l$ if $i=j$, $a_{ij}^l = b_j^l$ if $i>j$, and $a_{ij}^l = 0$ if $i<j$, which satisfies \eqref{3.2}. In contrast, any explicit Runge-Kutta methods cannot preserve the symplectic structure.
\end{rk}

\section{Main results}
In this section, we demonstrate that the two numerical methods proposed in Section 3 are structure-preserving.
\begin{tm}\label{th4.1}
Assuming that the implicit Runge-Kutta methods preserve the symplectic structure, the transformed Runge-Kutta methods for SPSs with non-constant structure matrix preserve both the Poisson structure and the Casimir functions. 
\end{tm}


The proof of the above theorem is omitted here. The transformed Runge-Kutta methods we propose here are a special case of the one presented in reference \cite{Hongetal2021}. 

In the following, we demonstrate that the diagonal implicit Runge-Kutta methods for SPSs preserve Casimir functions, quadratic Hamiltonian functions, and the Poisson structure.
\begin{tm}\label{th4.2}
Suppose that $B(y)$ is a constant structure matrix, i.e., $B(y) \equiv B$, then the diagonal implicit Runge-Kutta methods preserve the Casimir functions.
\end{tm}
\begin{proof}
It suffices to prove $C(y_{k+1})=C(y_k)$, that is
	\begin{align*}
		C(y_{k+1})-C(y_k)=&\nabla C(y_k^*)^\top \left(y_{k+1}-y_k \right)\\
		=&\nabla C(y_k^*)^\top \left(h\sum_{i=1}^sb_{i}^0f_0(Y_i)+\sum_{r=1}^m\sum_{i=1}^sb_{i}^rf_r(Y_i)J_k^r \right)\\
		=&\nabla C(y_k^*)^\top B \left(h\sum_{i=1}^sb_{i}^0\nabla H_0(Y_i)+\sum_{r=1}^m\sum_{i=1}^sb_{i}^r\nabla H_r(Y_i)J_k^r \right)\\
		=&0,
	\end{align*}
	where $y_k^*=y_k+\theta \left(y_{k+1}-y_k \right) \, \left(\theta \in (0,1) \right)$. Thus the conclusion of this theorem is proved.
\end{proof}
Considering the quadratic Hamiltonian functions $H_i(y) = \frac{1}{2}y^\top S_i y(i=1, \cdots ,s)$ with constant symmetric matrices $S_i$, we will prove that the diagonal implicit Runge-Kutta methods preserve these $H_i(y)$ under certain conditions. 
\begin{tm}\label{th4.3}
Suppose that $B(y)$ is a constant structure matrix, i.e., $B(y) \equiv B$,  and $\nabla H_i(y)^\top B \nabla H_j(y) = 0$, then the diagonal implicit Runge-Kutta methods preserve the quadratic Hamiltonian functions $H_i(y)(i=1, \cdots ,s)$.
\end{tm}
\begin{proof}
	We need to show that $H_i(y_{k+1})-H_i(y_k)=0$. According to the $H_i(y) = \frac{1}{2}y^\top S_i y$,  we obtain
	\begin{equation}\label{4.1}
			\begin{aligned}
			\frac{1}{2}y_{k+1}^\top S_iy_{k+1}=&\frac{1}{2}\left (y_k+h\sum_{i=1}^sb_{i}^0f_0(Y_i)+\sum_{r=1}^m\sum_{i=1}^sb_{i}^rf_r(Y_i)J_k^r\right)^\top S_i\\
			&\times\left (y_k+h\sum_{i=1}^sb_{i}^0f_0(Y_i)+\sum_{r=1}^m\sum_{i=1}^sb_{i}^rf_r(Y_i)J_k^r\right)
		\end{aligned}
	\end{equation}
Substituting $y_k=Y_i-h\sum_{j=1}^sa_{ij}^0f_0(Y_j)-\sum_{r=1}^m\sum_{j=1}^sa_{ij}^rf_r(Y_j)J_k^r$ into \eqref{4.1}, and combining with $\nabla H_i(y)^\top B \nabla H_j(y) = 0$, it follows that 
	\begin{align*}
			\frac{1}{2}y_{k+1}^\top S_iy_{k+1}=&\frac{1}{2}y_k^\top S_iy_k+\frac{1}{2}hY_i^\top S_i\left(\sum_{i=1}^sb_{i}^0f_0(Y_i)\right)+\frac{1}{2}Y_i^\top S_i\left(\sum_{r=1}^m\sum_{i=1}^sb_{i}^rf_r(Y_i)J_k^r\right)\\
			&+\frac{1}{2}h\left(\sum_{i=1}^sb_{i}^0f_0(Y_i)\right)^\top S_iY_i+\frac{1}{2}\left(\sum_{r=1}^m\sum_{i=1}^sb_{i}^rf_r(Y_i)J_k^r\right)^\top S_iY_i\\
			&+\frac{1}{2}h^2\left[b_i^0b_j^0- b_i^0a_{ij}^0- b_j^0a_{ji}^0\right]\left(\sum_{i=1}^sf_0(Y_i)\right)^\top S_i\left(\sum_{j=1}^sf_0(Y_j)\right)\\
			&+\frac{1}{2}h\left[b_i^0b_j^r- b_i^0a_{ij}^r- b_j^ra_{ji}^0\right]\left(\sum_{r=1}^m\sum_{j=1}^sf_r(Y_j)J_k^r\right)^\top S_i\left(\sum_{i=1}^sf_0(Y_i)\right)\\
			&+\frac{1}{2}h\left[b_i^0b_j^r- b_i^0a_{ij}^r- b_j^ra_{ji}^0\right]\left(\sum_{i=1}^sf_0(Y_i)\right)^\top S_i\left(\sum_{r=1}^m\sum_{j=1}^sf_r(Y_j)J_k^r\right)\\
			&+\frac{1}{2}\left[b_i^rb_j^\zeta- b_i^ra_{ij}^\zeta- b_j^\zeta a_{ji}^r\right]\left(\sum_{r=1}^m\sum_{i=1}^sf_r(Y_i)J_k^r\right)^\top S_i\left(\sum_{\zeta=1}^m\sum_{j=1}^sf_r(Y_j)J_k^r\right)\\
			=&\frac{1}{2}y_k^\top S_iy_k.
\end{align*}

\end{proof}

\begin{tm}\label{th4.4}
	Suppose that $B(y)$ is a constant structure matrix, i.e., $B(y) \equiv B$, then the diagonal implicit Runge-Kutta methods preserve the Poisson structure.
\end{tm}
\begin{proof}
Denote $X_i = \frac{\partial Y_i}{\partial y_k} (i = 1, \cdots, s)$, $D_i^l = Df_l(Y_i) (l = 0, \cdots, m; i = 1, \cdots, s)$, where $Df_l$ represents the derivative of the function $f_l$, then
		\begin{eqnarray}
			\frac{\partial y_{k+1}}{\partial y_k}&=&I+h\sum_{i=1}^sb_i^0D_i^0X_i+\sum_{r=1}^m\sum_{i=1}^sb_i^rJ_k^rD_i^rX_i, \label{4.2} \\
          \frac{\partial Y_i}{\partial y_k}&=&I+h\sum_{j=1}^sa_{ij}^0D_j^0X_j+\sum_{r=1}^m\sum_{j=1}^sa_{ij}^rJ_k^rD_j^rX_j, \label{4.3}
		\end{eqnarray}
	and
		\begin{equation}\label{4.4}
			B\left(D_i^l\right)^\top +D_i^lB=B(B\nabla^2H_l)^\top+(B\nabla^2H_l)B=0,
	\end{equation}
where the Hessian matrices $\nabla ^2H_l(l=0, \cdots, m)$ are symmetric. Additionally, according to \eqref{4.2}, we have 
\begin{equation}\label{4.5}
  \begin{aligned}
  	\frac{\partial y_{k+1}}{\partial y_k}B{\left(\frac{\partial y_{k+1}}{\partial y_k}\right)}^\top =&\left(I+h\sum_{i=1}^sb_i^0D_i^0X_i+\sum_{r=1}^m\sum_{i=1}^sb_i^rJ_k^rD_i^rX_i\right)B\\
  	&\times {\left(I+h\sum_{i=1}^sb_i^0D_i^0X_i+\sum_{r=1}^m\sum_{i=1}^sb_i^rJ_k^rD_i^rX_i\right)}^\top
  \end{aligned}
\end{equation}
Inserting \eqref{4.3} into \eqref{4.5}, it yields
	\begin{equation}\label{4.6}
		\begin{aligned}
			\frac{\partial y_{k+1}}{\partial y_k}B{\left(\frac{\partial y_{k+1}}{\partial y_k}\right)}^\top =&B+h\sum_{i=1}^sb_i^0[B{\left(D_i^0X_i\right)}^\top+D_i^0X_iB]\\
			&+\sum_{r=1}^m\sum_{i=1}^sb_i^rJ_k^r[B{\left(D_i^rX_i\right)}^\top+D_i^rX_iB]\\
			&+h^2\left( \sum_{i=1}^s b_i^0 D_i^0 X_i \right) B \left( \sum_{i=1}^s b_i^0 D_i^0 X_i \right)^\top\\
			&+h\left( \sum_{i=1}^s b_i^0 D_i^0 X_i \right) B \left( \sum_{r=1}^m\sum_{i=1}^sb_i^rJ_k^rD_i^rX_i \right)^\top\\
			&+ h\left( \sum_{r=1}^m\sum_{i=1}^sb_i^rJ_k^rD_i^rX_i \right) B\left( \sum_{i=1}^s b_i^0 D_i^0 X_i \right)^\top\\
			&+\left( \sum_{r=1}^m\sum_{i=1}^sb_i^rJ_k^rD_i^rX_i \right) B\left( \sum_{r=1}^m\sum_{i=1}^sb_i^rJ_k^rD_i^rX_i \right)^\top.
		\end{aligned}
	\end{equation}
Substituting \eqref{4.3} into $X_iB\left( D_i^lX_i \right)^\top$ and $D_i^lX_iB\left(X_i \right)^\top$, we get
	\begin{equation}\label{4.7}
		\begin{aligned}
			X_iB\left( D_i^lX_i \right)^\top&=B\left( D_i^lX_i \right)^\top+h\sum_{j=1}^sa_{ij}^0D_j^0X_jB\left( D_i^lX_i \right)^\top+\sum_{r=1}^m\sum_{j=1}^sa_{ij}^rJ_k^rD_j^rX_jB\left( D_i^lX_i \right)^\top,\\
			D_i^lX_iB\left(X_i \right)^\top&=D_i^lX_iB+h\sum_{j=1}^sa_{ij}^0D_i^lX_i\left( D_j^0X_j \right)^\top+\sum_{r=1}^m\sum_{j=1}^sa_{ij}^rJ_k^rD_i^lX_iB\left( D_j^rX_j \right)^\top.
		\end{aligned}
	\end{equation}
From \eqref{4.6} and \eqref{4.7}, we know that
	\begin{equation}\label{4.8}
		\begin{aligned}
			\frac{\partial y_{k+1}}{\partial y_k}B{\left(\frac{\partial y_{k+1}}{\partial y_k}\right)}^\top=&B+h\sum_{i=1}^sb_i^0[ X_iB\left( D_i^0X_i \right)^\top+D_i^0X_iB\left(X_i \right)^\top]\\
			&+\sum_{r=1}^m\sum_{i=1}^sb_i^rJ_k^r[X_iB\left( D_i^rX_i \right)^\top+D_i^rX_iB\left(X_i \right)^\top]\\
			&+h^2\sum_{i,j=1}^s[b_i^0b_j^0- b_i^0a_{ij}^0- b_j^0a_{ji}^0]D_j^0X_jB\left( D_i^0X_i \right)^\top\\
			&+h\sum_{r=1}^m\sum_{i,j=1}^s[b_i^0b_j^r- b_i^0a_{ij}^r- b_j^ra_{ji}^0]D_j^rX_jB\left( D_i^0X_i \right)^\top\\
			&+h\sum_{r=1}^m\sum_{i,j=1}^s[b_i^0b_j^r- b_i^0a_{ij}^r- b_j^ra_{ji}^0]J_k^rD_i^0X_iB\left(D_j^rX_j \right)^\top\\
			&+\sum_{r,\zeta=1}^m\sum_{i,j=1}^s[b_i^rb_j^\zeta- b_i^ra_{ij}^\zeta- b_j^\zeta a_{ji}^r]J_k^\zeta J_k^rD_j^\zeta X_jB\left(D_i^rX_i \right)^\top.
		\end{aligned}
	\end{equation}
Based on Remark \ref{rk2}, it is only necessary to prove that $X_iB\left( D_i^lX_i \right)^\top+D_i^lX_iB\left(X_i \right)^\top=0$. Without loss of generality, we consider the case of $m=1$. Combining \eqref{4.4}, we only need to prove that the diagonal implicit Runge-Kutta methods satisfy $X_iBX_i^T=B$ at every stage.
	\begin{itemize}
		\item When $s=1$, form the first equation of \eqref{3.10}, we have	
		\begin{equation*}
			X_1=I+\frac{b_1^0}{2}hD_1^0X_1+\frac{b_1^1}{2}J_k^1 D_1^1X_1 ,\\
		\end{equation*}
		therefore
		\begin{equation*}
			X_1=\left (I-\frac{b_1^0}{2}hD_1^0-\frac{b_1^1}{2}J_k^1 D_1^1\right)^{-1}\colon=K^{-1}.
		\end{equation*}
Additionally, by applying the associativity of matrix multiplication and the skew-symmetry of the Poisson bracket \eqref{2.4}, we derive 
		\begin{equation}\label{4.9}
				D_1^0B\left(D_1^0\right)^T=0 ,\quad D_1^1B\left(D_1^1\right)^T=0,
		\end{equation}
		and
		\begin{equation}\label{4.10}
				D_1^0B\left(D_1^1\right)^\top+D_1^1B\left(D_1^0\right)^\top=0.
		\end{equation}
		To prove $X_1BX_1^\top=B$, it is equivalent to proving $KBK^\top=B$. According to
		 \eqref{4.4}, \eqref{4.9} and \eqref{4.10}, 
		\begin{equation}\label{4.11}
			\begin{aligned}
				KBK^T=&\left (I-\frac{b_1^0}{2}hD_1^0-\frac{b_1^1}{2}J_k^1 D_1^1\right)B\left (I-\frac{b_1^0}{2}h\left(D_1^0\right)^T-\frac{b_1^1}{2}J_k^1 \left(D_1^1\right)^T\right)\\
				=&B-\frac{b_1^0}{2}h[B\left(D_1^0\right)^T+D_1^0B]-\frac{b_1^1}{2}J_k^1[B\left(D_1^1\right)^T+D_1^1B]+\frac{(b_1^0)^2}{4}h^2D_1^0B\left(D_1^0\right)^\top\\
				&+\frac{b_1^0b_1^1}{4}hJ_k^1[D_1^0B\left(D_1^1\right)^T+D_1^1B\left(D_1^0\right)^\top]+\frac{(b_1^1)^2}{4}J_k^1J_k^1D_1^1B\left(D_1^1\right)^\top\\
				=&B.
			\end{aligned}
		\end{equation}
	\item When $s = 2$, consider the second equation of \eqref{3.10} and let
	\begin{equation}\label{4.12}
		\begin{aligned}
			F(y_k, Y_1, Y_2)\colon=&Y_2 -y_k - b_1^0hf_0\left(Y_1 \right)-\frac{b_2^0}{2} h f_0\left(Y_2 \right) - b_1^1 f_1 \left(Y_1 \right) J_k^1-  \frac{b_2^1}{2} f_1 \left(Y_2 \right) J_k^1\\
			=&0.
		\end{aligned}
	\end{equation}
	Using the implicit function theorem, we obtain
	\begin{equation*}
		\frac{\partial Y_2}{\partial y_k}=- \left(\frac{\partial F}{\partial Y_2} \right)^{-1}\left(\frac{\partial F}{\partial y_k} \right).
	\end{equation*}
	Thus, to prove $X_2BX_2^T=B$ is equivalent to show
	\begin{equation}\label{4.13}
		\frac{\partial F}{\partial Y_2}B\left(\frac{\partial F}{\partial Y_2} \right)^\top=	\frac{\partial F}{\partial y_k}B\left(\frac{\partial F}{\partial y_k} \right)^\top.
	\end{equation}
	By \eqref{4.12}, we have
	\begin{equation*}
		\frac{\partial F}{\partial y_k}=-I,
	\end{equation*}
	\begin{equation*}
		\frac{\partial F}{\partial Y_2}=I-\frac{b_2^0}{2} hD_2^0-\frac{b_2^1}{2} J_k^1 D_2^1,
	\end{equation*}
	then
	\begin{equation*}
		\frac{\partial F}{\partial y_k}B\left(\frac{\partial F}{\partial y_k} \right)^\top=B.
	\end{equation*}
	Similar to the process of calculating \eqref{4.11}, there are
	\begin{equation*}
			\frac{\partial F}{\partial Y_2}B\left(\frac{\partial F}{\partial Y_2} \right)^\top=B.
	\end{equation*}
	Therefore, \eqref{4.13} holds, i.e., $X_2BX_2^\top=B$.
	\item When $s=3, 4, \cdots$, each stage $s$ of the diagonal implicit Runge-Kutta methods satisfies $X_iBX_i^\top = B$ similar to the case of $s=2$. To avoid tediousness, we omit the proof of these cases.
		\end{itemize}
	Therefore, we obtain that the diagonal implicit Runge-Kutta methods preserve the Poisson structure.
\end{proof}

\begin{rk}\label{rk4}
	The midpoint method preserves the Poisson structure when applied to SPSs with the constant structural matrix $B(y) \equiv B$. To illustrate this, consider the case with $m=1$, where the coefficients are set to $a_{11}^0 = a_{11}^1 = \frac{1}{2}$ and $b_1^0 = b_1^1 = 1$. The  diagonal implicit Runge-Kutta methods become to
	\begin{equation*}
		\begin{aligned}
			Y_1 &= y_k + h\frac{1}{2}f_0(Y_1) + \frac{1}{2}f_1(Y_1)J_k^r,\\
			y_{k+1} &= y_k + hf_0(Y_1) + f_1(Y_1)J_k^1,
		\end{aligned}
	\end{equation*}
and the proof process is identical to that of Theorem \ref{th4.4} when $s = 1$. In addition, reference \cite{Wangetal2022a} employs alternative proof methods to demonstrate that the midpoint method preserves the Poisson structure.
\end{rk}

\section{Numerical experiments}
In this section, we present some numerical examples to validate the structure-preserving properties of the two proposed numerical methods. In all experiments, the two-stage diagonal implicit symplectic Runge-Kutta methods are applied to the stochastic rigid body system and the linear stochastic Poisson system to verify preserving the Poisson structure, Casimir functions, and quadratic Hamiltonian functions.

\subsection{Stochastic rigid body systems}
\label{sub5.1}
Consider the following $3$-dimensional stochastic rigid body system:
\begin{equation}\label{5.1}
	\begin{aligned}
		dy(t)&= \begin{pmatrix}
			0 & -y_{3} & y_{2} \\
			y_{3} & 0 & -y_{1} \\
			-y_{2} & y_{1} & 0
		\end{pmatrix}\nabla H(y(t))\left(dt+c\circ dW(t)\right),\\
	y(0)&=y_0,
	\end{aligned}
\end{equation}
where $y=(y_1, y_2, y_3)^\top$, $y_0=(y_1^0, y_2^0, y_3^0)^\top $, and $c$ is a non-zero constant. The Hamiltonian function is given by $H(y)=\frac{1}{2}\left(\frac{y_{1}^{2}}{I_{1}}+\frac{y_{2}^{2}}{I_{2}}+\frac{y_{3}^{2}}{I_{3}}\right)$, here $I_1$, $I_2$, and $I_3$ are the moments of inertia. 

The system exhibits the Casimir function
\begin{equation*}
	C(y)=\frac{1}{2}\left(y_{1}^{2}+y_{2}^{2}+y_{3}^{2}\right) \equiv \frac{1}{2}\left(\left(y_{1}^{0}\right)^{2}+\left(y_{2}^{0}\right)^{2}+\left(y_{3}^{0}\right)^{2}\right)=: \mathcal{C}.
\end{equation*}

According to coordinate transformation $\bar{y}=\theta (y)$ (see \cite{Hongetal2021}), we get
\begin{equation}\label{5.2}
	\bar{y}_1 = y_2, \quad
	\bar{y}_2 = \arctan \left(\frac{y_3}{y_1}\right), \quad \bar{y}_3 = \mathcal C,
\end{equation}
by solving the equations:
\begin{equation*}
	\begin{split}
		&\{\bar y_1,\bar y_1\}=0,\quad \{\bar y_1,\bar y_2\}=-1, \quad \{\bar y_1,\bar y_3\}=0,\\
		&\{\bar y_2,\bar y_1\}=1,\quad \{\bar y_2,\bar y_2\}=0, \quad \{\bar y_2,\bar y_3\}=0,\\
		&\{\bar y_3,\bar y_1\}=0,\quad \{\bar y_3,\bar y_2\}=0, \quad \{\bar y_3,\bar y_3\}=0.
	\end{split}
\end{equation*}

Defining $(\bar y_1, \bar y_2)=(P, Q)$, we obtain the stochastic Hamiltonian system
\begin{align}\label{5.3}
	d\begin{pmatrix}
		P\\
		Q
	\end{pmatrix}=
	\begin{pmatrix}
		0& -1\\
		1& 0
	\end{pmatrix}
	\nabla K(P, Q)(dt+c\circ dW(t)),
\end{align}
where $$K(P, Q)=
\frac{1}{2I_1} (2\mathcal C-P^2)\cos^2 (Q)
+\frac{1}{2I_2} P^2
+\frac{1}{2I_3} (2\mathcal C-P^2)\sin^2 (Q).$$

Taking $s=2$ in \eqref{3.6} and then applying it to \eqref{5.3}, we derive the following symplectic method
\begin{equation*}
	\begin{split}
		P_{n+1}=&P_n+h\sum_{i=1}^sb_i^0f(p_i, q_i)+\sum_{i=1}^sb_i^1cf(p_i, q_i)\Delta W_n,\\
		Q_{n+1}=&Q_n+h\sum_{i=1}^sb_i^0g(p_i, q_i)+\sum_{i=1}^sb_i^1cg(p_i, q_i)\Delta W_n,
	\end{split}
\end{equation*}
for $n=1, \cdots, N-1$ with $P_0=p,$ $Q_0=q$ and
\begin{equation*}
	\begin{split}
		p_i &= P_n+h\sum_{j=1}^sa_{ij}^0f(p_j, q_j)+\sum_{j=1}^sa_{ij}^1cf(p_j, q_j)\Delta W_n, \\
		q_i &= Q_n+h\sum_{j=1}^sa_{ij}^0g(p_j, q_j)+\sum_{j=1}^sa_{ij}^1cg(p_j, q_j)\Delta W_n,
	\end{split}
\end{equation*}
for $i, j=1, 2$, where the coefficients of the symplectic method are taken from Butcher tableau

\begin{align}\label{5.4}
	\renewcommand{\arraystretch}{1.25}
	\setlength{\tabcolsep}{20pt} 
	\begin{array}{|ccccc}
		\frac{1}{8} & 0 & \frac{1}{4}  & 0  \\
		\frac{1}{4} & \frac{3}{8} & \frac{1}{2}  & \frac{1}{4}  \\
		\hline
		\frac{1}{4} & \frac{3}{4} & \frac{1}{2} & \frac{1}{2} \\
	\end{array},
\end{align}
and
\begin{equation*}
	\begin{split}
		f(p, q) &= -\left(\frac{1}{2I_3}-\frac{1}{2I_1}\right)\left(2\mathcal{C}-p^2\right)sin(2q), \\
		g(p, q) &= \left(\frac{1}{I_2}-\frac{cos^2(q)}{I_1}-\frac{sin^2(q)}{I_3}\right)p.
	\end{split}
\end{equation*}

To ensure the solvability of implicit methods, it is necessary to appropriately truncate the Wiener increment $\Delta W_n=\sqrt{h}\xi_n$, $\xi_n\sim \mathcal{N}(0, 1)$. Here, we use the truncation $\Delta \widehat{W_n} := \sqrt{h}\zeta_n$, where
\[
\zeta_n =
\begin{cases}
	\xi_n & \text{if } |\xi_n| \leq A_h, \\
	A_h & \text{if } \xi_n > A_h, \\
	-A_h & \text{if } \xi_n < -A_h,
\end{cases} 
\]
and $A_h = \sqrt{2k|\ln h|}$ for $k \geq 1$. This truncation method is derived from \cite{Milsteinetal2002a}.

Then by using the inverse transformation of \eqref{5.2}, we obtain the transformed Runge-Kutta methods $\{y_n^1, y_n^2, y_n^3\}$ suitable for the original stochastic rigid system \eqref{5.1}, namely,
\begin{align}
	\label{5.5}
	y^1_n= \sqrt{2 \mathcal C - P_n^2} \cos (Q_n), \quad
	y^2_n= P_n, \quad
	y^3_n = \sqrt{2 \mathcal C - P_n^2} \sin (Q_n).
\end{align}

Next, the reference solutions are approximated by the midpoint method and can be expressed in the following component form:
\begin{equation}\label{5.6}
	\begin{pmatrix}
		y_{n+1}^1 \\
		y_{n+1}^2 \\
		y_{n+1}^3
	\end{pmatrix}
	= \begin{pmatrix}
		y_{n}^1 \\
		y_{n}^2 \\
		y_{n}^3
	\end{pmatrix}+\begin{pmatrix}
		0 & -\frac{y_{n}^3+y_{n+1}^3}{2} & \frac{y_{n}^2+y_{n+1}^2}{2} \\
		\frac{y_{n}^3+y_{n+1}^3}{2} & 0 & -\frac{y_{n}^1+y_{n+1}^1}{2} \\
		-\frac{y_{n}^2+y_{n+1}^2}{2} & \frac{y_{n}^1+y_{n+1}^1}{2} & 0
	\end{pmatrix}	\begin{pmatrix}
		\frac{y_{n}^1+y_{n+1}^1}{2I_1} \\
		\frac{y_{n}^2+y_{n+1}^2}{2I_2} \\
		\frac{y_{n}^3+y_{n+1}^3}{2I_3}
	\end{pmatrix}\left(h+c \Delta \widehat{W_n}\right).
\end{equation}

For our experiments, we choose $I_1 = \sqrt2 + \sqrt\frac{2}{1.51}$, $I_2 = \sqrt2 - 0.51\sqrt\frac{2}{1.51}$, $I_3 = 1$, $c= 0.2$, $k=4$, and initial values $y_0=(\frac{1}{\sqrt2}, \frac{1}{\sqrt2}, 0)^\top $. Here, we take $h = 0.01$ as the step size of numerical solutions and use a very small step size of $10^{-5}$ for the reference solutions. 
 
From Figure \ref{pp1}, we observe that the single sample paths for the coordinate components arising from the transformed Runge-Kutta methods coincide well with the reference solutions. Figure \ref{pp2} displays the phase orbit along one sample of the transformed Runge-Kutta methods and the reference solutions. It shows that the phase orbits of both the transformed Runge-Kutta methods and the reference solutions coincide completely. Figure \ref{pp3} presents the Casimir function $C(y_n)=\frac{1}{2}\left((y_n^1)^2+(y_n^2)^2+(y_n^3)^2\right)$ computed by the transformed Runge-Kutta methods and the reference solutions. It is shown that the transformed Runge-Kutta methods can preserve the Casimir function. Figure \ref{pp4} numerically presents that the root mean-square order of the transformed Runge-Kutta methods is 1. Here, the time steps considered are $h=[0.005, 0.01, 0.02, 0.04]$, and 500 samples are taken to approximate the expectation.

\begin{figure}[H]
	\centering
	\includegraphics[width=4in]{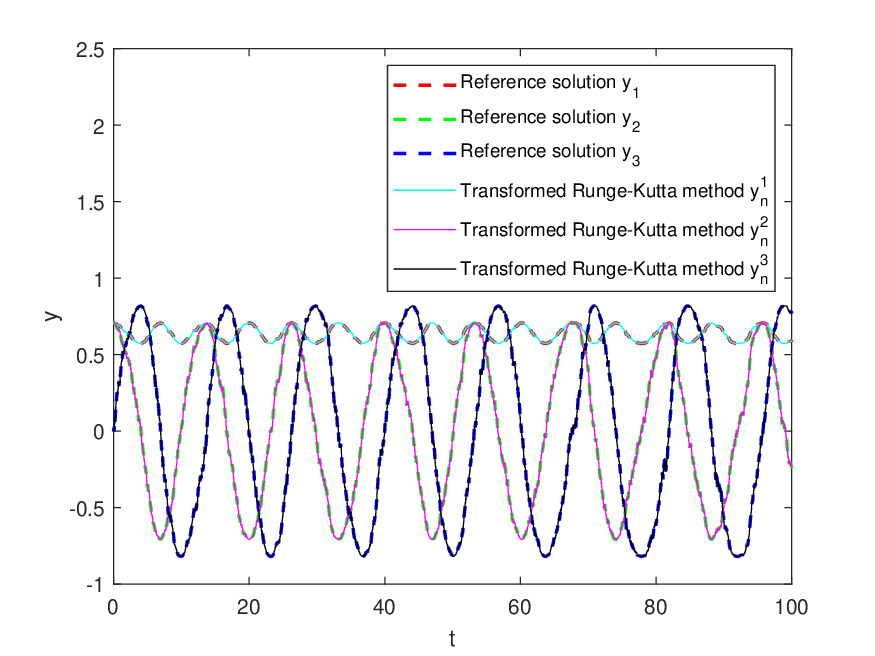}
	\caption{Sample paths produced by \eqref{5.5} and \eqref{5.6}.}\label{pp1}
\end{figure}

\begin{figure}[H]
	\centering
	\includegraphics[width=4in]{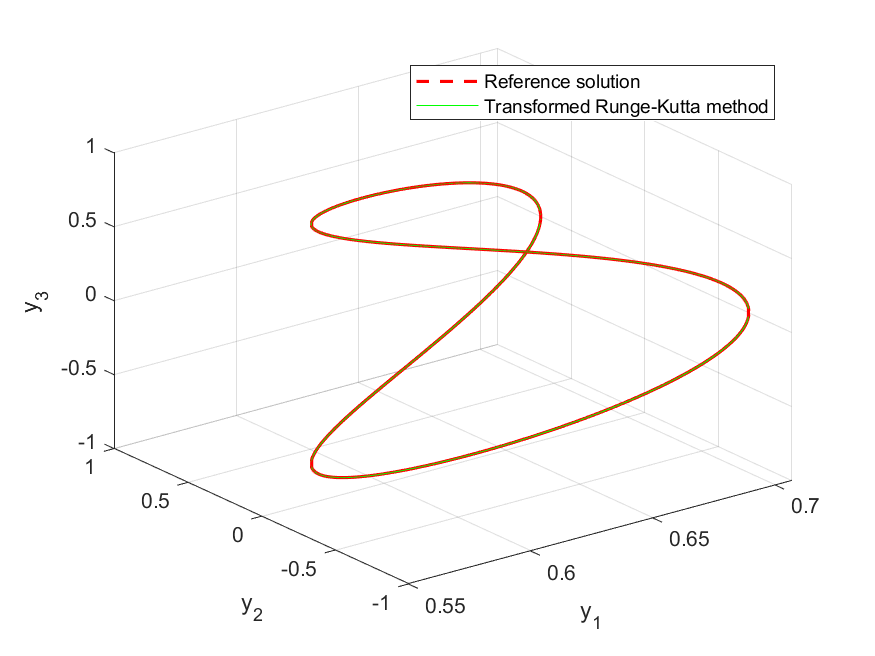}
	\caption{Phase orbits produced by \eqref{5.5} and \eqref{5.6}.}\label{pp2}
\end{figure}

\begin{figure}[H]
	\centering
	\includegraphics[width=4in]{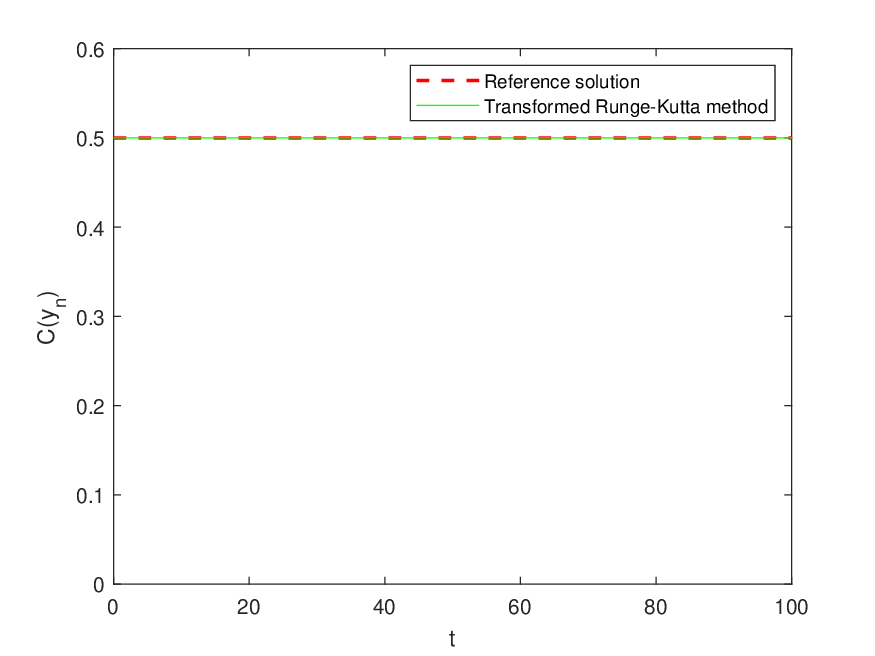}
	\caption{Casimir functions produced by \eqref{5.5} and \eqref{5.6}.}\label{pp3}
\end{figure}

\begin{figure}[H]
	\centering
	\includegraphics[width=4in]{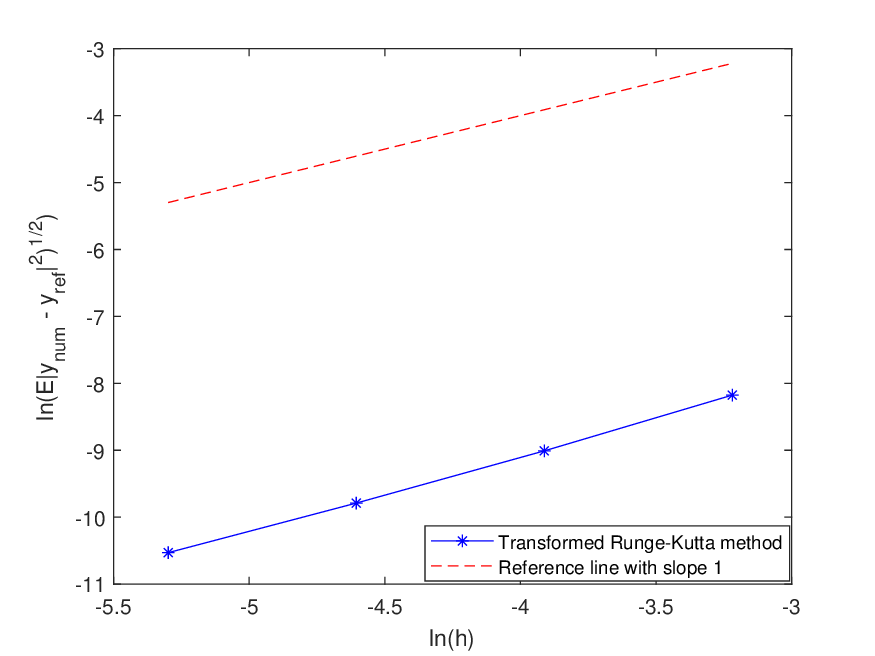}
	\caption{ Mean-square order of the transformed Runge-Kutta methods.}\label{pp4}
\end{figure}

\subsection{Linear stochastic Poisson systems} 
\label{sub5.2}
Now we consider the following linear stochastic Poisson system
\begin{equation}\label{5.7}
	\begin{aligned}
		dy(t) &= \begin{pmatrix}
			0 & 1 & -1\\
			-1 & 0 & 3  \\
			1 & -3 & 0\\
		\end{pmatrix} \left[\begin{pmatrix}
			2 & 1 & 1\\
			1 & 1 & 0  \\
			1 & 0 & 1\\
		\end{pmatrix}ydt + \frac{1}{4}\begin{pmatrix}
			11 & 4 & 4\\
			4 & 2 & 1  \\
			4 & 1 & 2\\
		\end{pmatrix}y\circ dW(t)\right],\\
		y(t_0) &= y_0.
	\end{aligned}
\end{equation}
where $y=(y_1, y_2, y_3)^\top$ and the Hamiltonians
 \begin{equation*}
 	H_1(y)=\frac{1}{2}y^\top\begin{pmatrix}
 		2 & 1 & 1\\
 		1 & 1 & 0  \\
 		1 & 0 & 1\\
 	\end{pmatrix}y, \quad H_2(y)=\frac{1}{8}y^\top\begin{pmatrix}
 		11 & 4 & 4\\
 		4 & 2 & 1  \\
 		4 & 1 & 2\\
 	\end{pmatrix}y.
 \end{equation*}
Obviously, $\left\{H_1, H_2\right\}(y)=\left(\nabla H_1(y) \right)^\top B(y) \nabla H_2(y)=0$, where 
\begin{equation*}
	B(y)\equiv \begin{pmatrix}
		0 & 1 & -1\\
		-1 & 0 & 3  \\
		1 & -3 & 0\\
	\end{pmatrix}.
\end{equation*}
This shows that both $H_1(y)$ and $H_2(y)$ are invariants of the system \eqref{5.7}. Additionally, we can obtain that the Casimir function of the system \eqref{5.7} is $C(y)=3y_{1}+y_{2}+y_{3}$, which is also an invariant of the system \eqref{5.7}.

Let
 \begin{align*}
	A_0&= \begin{pmatrix}
		0 & 1 & -1\\
		-1 & 0 & 3  \\
		1 & -3 & 0\\
	\end{pmatrix} 
   \begin{pmatrix}
		2 & 1 & 1\\
		1 & 1 & 0  \\
		1 & 0 & 1\\
	\end{pmatrix} , &
	A_1&= \frac{1}{4}\begin{pmatrix}
		0 & 1 & -1\\
		-1 & 0 & 3  \\
		1 & -3 & 0\\
	\end{pmatrix}
\begin{pmatrix}
		11 & 4 & 4\\
		4 & 2 & 1  \\
		4 & 1 & 2\\
	\end{pmatrix},
\end{align*}
the equivalent formulation of \eqref{5.7} is given by
\begin{equation}\label{5.8}
	\begin{aligned}
		dy(t) &= A_0 y dt+A_1 y \circ dW(t),\\
		y(t_0) &= y_0,
	\end{aligned}
\end{equation}
and the exact solution of the system \eqref{5.8} from \cite{Arnold1974} is
\begin{equation}\label{5.9}
		y(t) = \exp\left[(t-t_0)A_0+(W(t)-W(t_0))A_1\right]y(t_0).
\end{equation}

Next, by substituting \eqref{5.4} into \eqref{5.8}, we obtain 
\begin{equation}
	\label{5.10}
		y_{n+1}=y_n+\frac{1}{4}hA_0Y_1+\frac{1}{2}A_1Y_1J_k^1+\frac{3}{4}hA_0Y_2+\frac{1}{2}A_1Y_2J_k^1,
\end{equation}
and
\begin{equation*}
	\begin{split}
		Y_1 &= y_n+\frac{1}{8}hA_0Y_1+\frac{1}{2}A_1Y_1J_k^1, \\
		Y_2 &= y_n+\frac{1}{4}hA_0Y_1+\frac{1}{2}A_1Y_1J_k^1+\frac{3}{8}hA_0Y_2+\frac{1}{4}A_1Y_2J_k^1.
	\end{split}
\end{equation*}

In Figure \ref{pp5} and Figure \ref{pp6}, we choose $T = 10$, initial value $y_0 = (1, 0, -1)^\top$ and time step $h = 0.01$. We can see from Figure \ref{pp5} that the single sample paths for the coordinate components obtained from the diagonal implicit Runge-Kutta methods coincide with the exact solutions. Figure \ref{pp6} presents one sample phase orbit generated by the diagonal implicit Runge-Kutta methods and the exact solutions. The graphic shows a high degree of coincidence between the phase orbits. Figure \ref{pp7} illustrates the evolution of the Casimir function $C(y_n)=3y_n^1+y_n^2+y_n^3$ generated by the diagonal implicit Runge-Kutta methods and the exact solutions with $T = 10$, $h = 0.1$ and $y_0 = (1, 1, 2)^\top$. It demonstrates that our numerical methods effectively preserve the Casimir function. Figure \ref{pp8} shows the evolution of the Hamiltonians $H_1(y(t))$ and $H_2(y(t))$ produced by the diagonal implicit Runge-Kutta methods and the exact solutions with initial value $y_0 = (1, 1, 1)^\top$. It can be seen that our methods effectively preserve the quadratic Hamiltonian functions. Finally, Figure \ref{pp9} shows that the mean-square convergence order of the diagonal implicit Runge-Kutta methods is 1. 

\begin{figure}[H]
	\centering
	\includegraphics[width=4in]{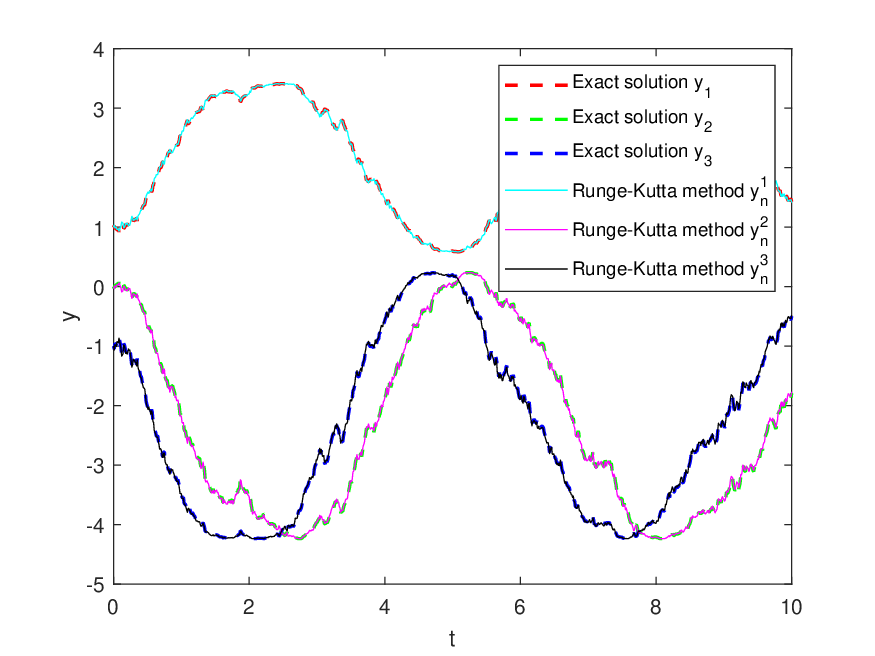}
	\caption{Sample paths produced by \eqref{5.9} and \eqref{5.10}.}\label{pp5}
\end{figure}

\begin{figure}[H]
	\centering
	\includegraphics[width=4in]{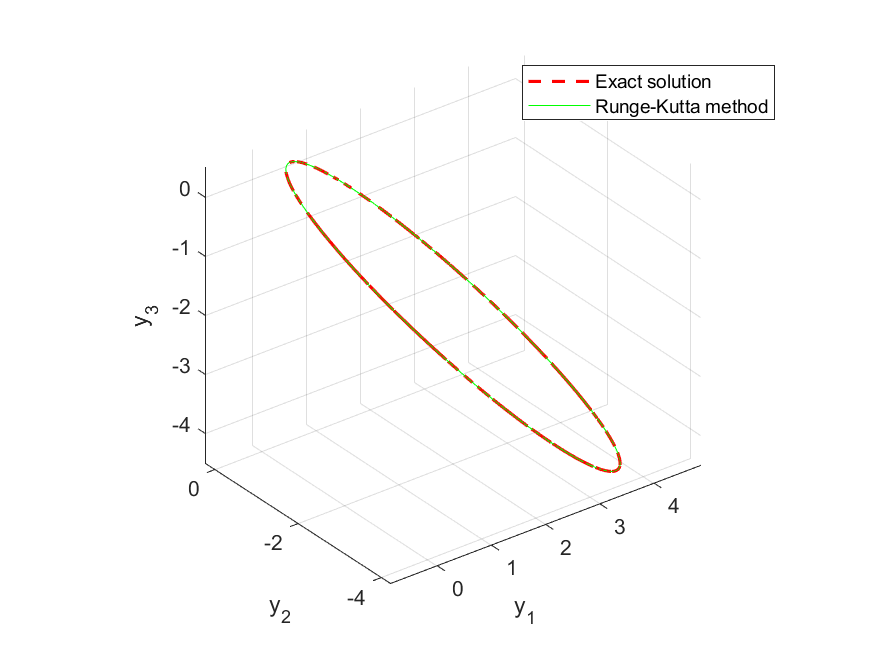}
	\caption{Phase orbits produced by \eqref{5.9} and \eqref{5.10}.}\label{pp6}
\end{figure}

\begin{figure}[H]
	\centering
	\includegraphics[width=4in]{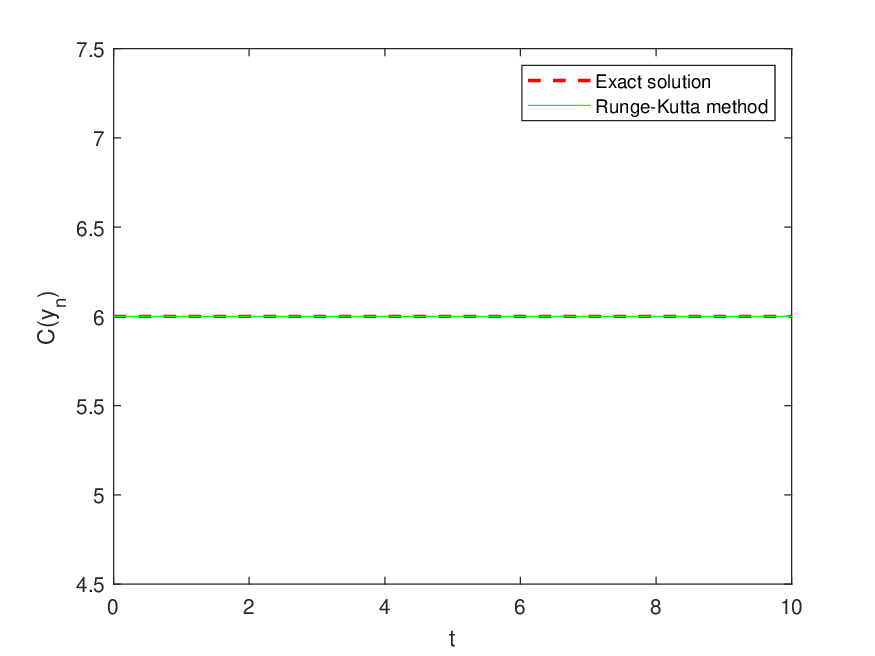}
	\caption{Casimir functions produced by \eqref{5.9} and \eqref{5.10}. }\label{pp7}
\end{figure}

\begin{figure}[H]
	\centering
	\includegraphics[width=4in]{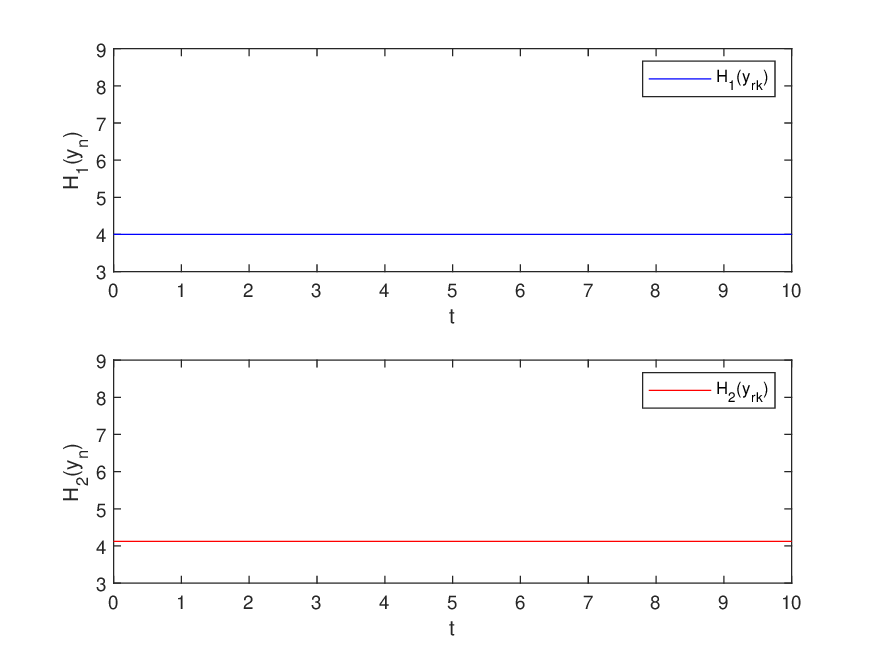}
	\caption{Evolution of $H_1(y(t)),$ $H_2(y(t))$ by \eqref{5.10}. }\label{pp8}
\end{figure}

\begin{figure}[H]
	\centering
	\includegraphics[width=4in]{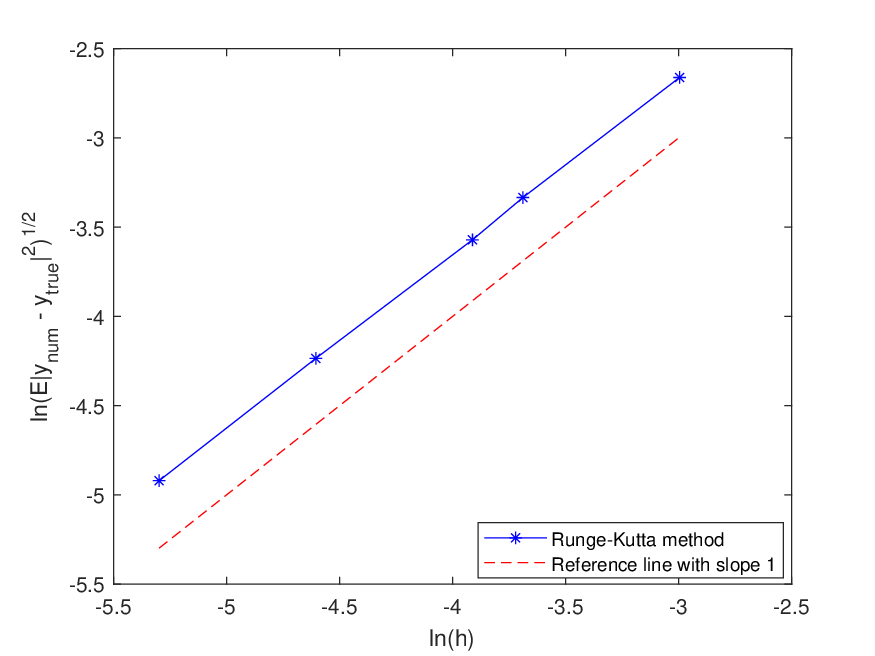}
	\caption{Mean-square order of the diagonal implicit Runge-Kutta methods with $h=[0.005, 0.01, 0.02, 0.025, 0.05]$ and sample paths 1000. }\label{pp9}
\end{figure}


\section{Conclusion}
This paper investigates the structure-preserving properties of implicit Runge-Kutta methods for stochastic Poisson systems driven by multiple noises. For different forms of the structure matrix $B(y)$, we propose the transformed Runge-Kutta methods and the diagonal implicit Runge-Kutta methods, and systematically analyse the structure-preserving properties of these numerical methods. Numerical experiments demonstrate the effectiveness and reliability of the proposed methods. 
Future research will focus on optimizing the coefficient conditions for more complex classes of Runge-Kutta methods, with the aim of exploring their structure-preserving properties in high-dimensional stochastic Poisson systems. Furthermore, the methods proposed in this paper can be extended to stochastic partial differential equations to broaden their applicability in more scientific fields of stochastic computing.

\section*{Acknowledgments} 
The authors would like to express their appreciation to the referees for their useful comments and the editors. Liying Zhang is supported by the National Natural Science Foundation of China (No.\,11601514 and No.\,11971458), the Fundamental Research Funds for the Central Universities (No.\,2023ZKPYL02 and No.\,2023JCCXLX01) and the Yueqi Youth Scholar Research Funds for the China University of Mining and Technology-Beiiing (No.\,2020YQLX03). Lijin Wang is supported by the National Natural Science Foundation of China (No.\,11971458).

\bibliography{references}
\bibliographystyle{plain}


\end{document}